\documentclass[10pt]{article}
\usepackage{amsmath,amssymb,amsthm,mathrsfs}\usepackage[pdftex]{graphicx}
\usepackage{fullpage}
\usepackage{cite}
\usepackage{palatino}

\usepackage{tikz}
\usetikzlibrary{matrix}
\usetikzlibrary{arrows,decorations.markings}

\newcommand{\R}{\mathbb{R}}

\newcommand{\e}{\epsilon}
\renewcommand{\phi}{\mathbf f}

\newcommand{\eh}[2]{e_{{#1},{#2}}^{\mathsf{H}}}
\newcommand{\ev}[2]{e_{{#1},{#2}}^{\mathsf{V}}}
\newcommand{\ed}[2]{e_{{#1},{#2}}^{\mathsf{D}}}

\newcommand{\sir}[2]{\sigma_{{#1},{#2}}^{\mathsf{R}}}
\newcommand{\sil}[2]{\sigma_{{#1},{#2}}^{\mathsf{L}}}

\newcommand{\sv}[3]{({#1},{#2},\mathsf{#3})}

\newcommand{\tsv}[3]{\theta_{\sv{#1}{#2}{#3}}}

\newtheorem{thm}{Theorem}[section]
\newtheorem{lem}[thm]{Lemma}
\newtheorem{prop}[thm]{Proposition}

\newtheorem{definition}[thm]{Definition}
\newtheorem{rem}[thm]{Remark}
\newtheorem{cor}[thm]{Corollary}

\newcommand{\av}[1]{\left|{#1}\right|}
\newcommand{\ip}[2]{\left\langle{{#1}},{{#2}}\right\rangle}
\newcommand{\mc}[1]{\mathcal{#1}}

\newcommand{\dmax}{d_{{\mathsf {max}}}}
\newcommand{\uppadj}{\smallfrown}

\newcommand{\lowadj}{\smallsmile}

\newcommand{\bdy}{\partial}

\DeclareMathOperator{\diag}{diag}
\DeclareMathOperator{\im}{Im}

\title{Consensus on simplicial complexes, or: \\The nonlinear simplicial Laplacian}
\author{Lee DeVille\\University of Illinois}

\begin{document}

\maketitle

\begin{abstract}
  We consider a nonlinear flow on simplicial complexes related to the simplicial Laplacian, and show that it is a generalization of various consensus and synchronization models commonly studied on networks.  In particular, our model allows us to formulate flows on simplices of any dimension, so that it includes edge flows, triangle flows, etc.   We show that the system can be represented as the gradient flow of an energy functional, and use this to deduce the stability of various steady states of the model.  Finally, we demonstrate that our model contains higher-dimensional analogues of structures seen in related network models.
\end{abstract}

{\bf Keywords:} Simplicial Laplacian, combinatorial Laplacian, consensus, synchronization, Kuramoto

{\bf AMS classification:} 34D06, 55U10, 15A18, 68M14, 05C65
\section{Introduction}

\subsection{Background and motivation}

As the field of network science has grown over the last few decades, there has been a significant interest in the mathematical community on understanding dynamical systems on networks~\cite{newman2006structure, boccaletti2006complex, barabasi2016network, newman2018networks}; by this it is generally meant a system that is defined by relatively simple subsystems that are pairwise coupled in various ways.  Of particular interest have been systems where  individual units evolve under relatively simple dynamics in isolation, but the coupling  induces an emergent dynamics on the network level.

More recently, there has been a growing recognition that networks are not enough;  we should not just consider pairwise interactions, but  we should also pay attention to interactions of larger subsets of elements~\cite{salnikov2018simplicial, mulder2018network, lambiotte2019networks, torres2020and, serrano2020simplicial}.  This viewpoint has lead to breakthroughs in several application domains, including neuroscience~\cite{giusti2016two, bassett2017network, curto2017can, sizemore2018cliques, kanari2018topological}, biological processes~\cite{klamt2009hypergraphs, ladroue2009beyond, roman2015simplicial}, statistical physics~\cite{courtney2016generalized}, social networks~\cite{patania2017shape, iacopini2019simplicial} and signal analysis/prediction~\cite{muhammad2006control, stolz2017persistent, benson2018simplicial, barbarossa2020topological}. For  very nice and comprehensive surveys of the state of the art, see~\cite{battiston2020networks, porter2020nonlinearity+}.  

Another way of describing the above is that we should not just consider graph structures but also hypergraph structures~\cite{berge1984hypergraphs, bollobas1986combinatorics} in dynamics.  One special class of hypergraphs is those that form {\em simplicial complexes}; these are hypergraphs with the property that if a subset is active, then all of its subsets are also active.  The study of simplicial complexes has a long history in mathematics~\cite{Munkres.book, Hatcher.book, lim2020hodge} and underpin many modern tools in algebraic topology.   The recent field of {\em topological data analysis}~\cite{carlsson2009topology, patania2017topological, speidel2018topological, bubenik2020persistence} has the goal of using tools from algebraic topology to understand complex datasets --- in particular, the study of the persistence of certain simplicial complexes as a function of various parameters has given significant insight in many applications.

One class of networks that have been broadly considered on networks are called {\em synchronization} or {\em consensus} dynamics; as a general rule, the former tends to refer to the case where the microscopic dynamics are given by an oscillator~\cite{peskin1975mathematical, Kuramoto.75, Sastry.Varaiya.80, Sastry.Varaiya.81, Ermentrout.1985,S, Kuramoto.91, Kuramoto.book, pikovsky2003synchronization, SM1, Acebron.etal.05, Mirollo.Strogatz.05, Wiley.Strogatz.Girvan.06, Abrams.Mirollo.Strogatz.Wiley.08, Arenas.etal.08, Dorfler.Bullo.2011, Dorfler.Bullo.12, Dorfler.Chertkov.Bullo.13, Dorfler.Bullo.14, Bronski.DeVille.Ferguson.16, Delabays.Coletta.Jacquod.16, Delabays.Coletta.Jacquod.17, Bronski.Ferguson.2018, ferguson2018volume, ferguson2018topological, pikovsky2003synchronization, salova2020decoupled}, the latter when the miscroscopic dynamics live in some Euclidean or metric space~\cite{degroot1974reaching, aumann1976agreeing, hegselmann2002opinion, Boyd.Diaconis.Xiao.04, Xiao.Boyd.04, ren2005survey, Xiao.Boyd.Kim.07, olfati2007consensus, wang2010finite, mesbahi2010graph, srivastava2011bifurcations, neuhauser2020multibody}.  The common thread of this class of models is to understand how and when the units work in concert to form some sort of coherent structure; for example, this might mean that a collection of oscillators will all be rotating with the same frequency, or it might mean that individual agents are converging to a commonly shared opinion.

Many of the linear synchronization/consensus models on networks are related to the graph Laplacian or its generalizations~\cite{bronski2014spectral, homs2020nonlinear, lucas2020multiorder}.  As such, it is  natural to consider nonlinear dynamics whose linearizations correspond to an appropriate Laplacian.   Dynamics in this class are widely studied on networks, and have been studied as well in both hypergraphical~\cite{skardal2019higher, xu2020spectrum, lucas2020multiorder, xu2020bifurcation, skardal2020memory, landry2020effect} and simplicial contexts~\cite{millan2020explosive,  horstmeyer2020adaptive, hansen2020opinion}.  (In particular, the model in~\cite{millan2020explosive} is quite similar to ours and fits into the class that we study, see Section~\ref{sec:outtro}.) 

The study of the Laplacian that corresponds to simplicial complexes has a long history, see~\cite{eckmann1944harmonische, friedman1998computing, dong2002combinatorial, duval2002shifted, horak2013spectra, horak2013interlacing} and has been termed the {\em combinatorial Laplacian}; here we will also call it the {\em simplicial Laplacian} for clarity.  The model presented in this paper is a nonlinear generalization of the simplicial Laplacian, both in the sense that its linearization recovers the Laplacian, and also that it its construction involves the introduction of a nonlinearity intertwined with the process of constructing the simplicial  Laplacian.   There has also been work in understanding random walks on simplicial complexes and their relationship to Laplacians and harmonic functions~\cite{parzanchevski2017simplicial, schaub2020random}.

\subsection{Summary of results of this paper}

In Section~\ref{sec:model} we present the model that we study in this paper; in Section~\ref{sec:def} we give a review of the definition and properties of simplicial complexes and give the system of equations that we study, in Section~\ref{sec:connection} we show that our model is a generalization of standard synchronization/consensus models such as the Kuramoto model, in Section~\ref{sec:interp} we give a coordinate-wise description of our model.  The main results of the paper are in Section~\ref{sec:theory}: in Section~\ref{sec:gradient} we show that the nonlinear simplicial Laplacian can be written as a (rescaled) gradient flow; in Section~\ref{sec:linearization} we compute the linearization of the nonlinear simplicial Laplacian around an arbitrary steady state and derive some of its properties; in Section~\ref{sec:convexity} we use convexity to establish the nonlinear stability of certain steady states; and in Section~\ref{sec:compactness} we consider the question of the existence of steady states at all.  Finally, in Section~\ref{sec:examples} we give several examples of the flows that arise under various choices of simplicial complexes, and demonstrate that the simplicial dynamics have many of the complex features observed in network dynamics.

\section{The model}\label{sec:model}

\subsection{General model definition}\label{sec:def}

As described above, we would like to consider a nonlinear generalization of the simplicial Laplacian that can subsumes general nonlinear consensus algorithms.  We first give an overview of the definitions relating to simplicial complexes, and then describe the flow we study.

\begin{definition}[Simplicial complex]
Let $V$ be a set.  A {\bf $k$-simplex} is a unordered set $\{v_0,\dots,v_k\}$ with $v_i\in V$ and $v_i\neq v_j$.  A {\bf face} of a $k$-simplex is all $(k-1)$-simplices of the form $\{v_0,\dots,v_k\}\setminus \{v_i\}$, which we will also denote by $\{v_0,\dots,\widehat{v_i},\dots,v_k\}$.  A {\bf simplicial complex} $X$ is a collection of simplices closed under inclusion of all faces, and we denote $X_d$ as the set of simplices strictly of dimension $d$.  We denote by $\dmax(X)$ as the maximal dimension of all of the simplices in $X$.

Two $d$-simplices $F_1,F_2$ of a complex $X$ are called {\bf upper adjacent} if both are faces of some $d+1$-simplex; in this case we write $F_1 \uppadj F_2$. 
Two $d$-simplices $F_1,F_2$ of a complex $X$ are called {\bf lower adjacent} if they have a common face; in this case we write $F_1 \lowadj F_2$. 
\end{definition}

\begin{definition}[Ordered and weighted complexes]
An ordering of vertices gives an {\bf orientation}, and an {\bf oriented $k$-simplex} is a $k$-simplex with an order type.  By definition, orientation is anti-symmetric with respect to an exchange of vertices.  We denote an ordered simplex by square brackets, e.g. $[v_0,\dots, v_k]$.  

Finally, we also allow for {\bf weighted} simplices; for each simplex in the complex we can attach a real number, known as the {\bf weight} of the simplex.  For a simplex $F\in X$ we denote the weight of $F$ by $w(F)$.  We will assume throughout that all weights are positive.
\end{definition}

For $X$ a simplicial complex, let $C_d(X)$ be the $\R$-vector space with basis given by the elements of $X_d$.  For each $d$, we have the {\bf boundary map} defined by 
\begin{align*}
  \partial_d \colon C_d(X) &\to C_{d-1}(X)\\
    [v_0,\dots,v_k]&\mapsto \sum_{\ell=0}^k (-1)^\ell [v_0,\dots,\widehat{v_\ell},\dots, v_k]
\end{align*}
and extended by linearity.  (As is well known, the map $\partial_d\circ \partial_{d+1}$ is identically zero.)  The kernel of the map $\partial_d\colon C_d(X)\to C_{d-1}(X)$ is known as the $d$-cycles of the complex $X$, and the image of $\partial_{d+1}\colon C_{d+1}(X)\to C_d(X)$ are called the $d$-boundaries.  Since $\partial_d\circ \partial_{d+1}=0$, it is natural to define the $d$-th homology group $ H_d(X) := \ker \partial_d / \im \partial_{d+1}$.

The weights induce a function $w\colon C_d(X)\to C_d(X)$ generated by the map that attaches $w(F)$ to each $F$. We will want to use matrix representations of the maps $\partial$ and $w$ and as such we need to choose a  basis.  We assume throughout that the basis chosen is always given by the elements of $X_d$, so that the choice of basis corresponds only to an ordering of the simplices in $X_d$.  Once we have made this choice, we let $B_d$ be the matrix representation of $\partial_d$ and $W_d$ the matrix representation of $w$.  (Again, this matrix will depend on the choice of ordering of the simplices in $X_d$, but nothing else.)  If $\theta_d\in C_d(X)$, we have   
\begin{equation*}
  B_d \theta_d \in C_{d-1}(X),\quad B_{d+1}^\intercal\theta_d \in C_{d+1}(X),\quad W_d\theta_d \in C_d(X).
\end{equation*}
(We give some examples below in Section~\ref{sec:interp}.)  Note that $W_d$ is always a diagonal matrix with the weights given on the diagonal.

Let $\phi\colon\R\to\R$ be a generic function with $\phi(0) = 0$ and $\phi'(0)>0$.  This induces a function $\phi\colon \R^n\to\R^n$ defined by  $\phi(x)_i = \phi(x_i)$ where we apply the scalar function $\phi(\cdot)$ componentwise.

For each $d$, let $\omega_d\in C_d(X)$ and we can then define 
the system of equations:
\begin{equation}\label{eq:nl}
\begin{split}
	\frac{d}{dt} \theta_0 &= \omega_0-W_0^{-1}B_{1}W_{1}\phi(B_{1}^\intercal \theta_0),\\
  \frac{d}{dt} \theta_d &= \omega_d-\left[W_d^{-1}B_{d+1}W_{d+1}\phi(B_{d+1}^\intercal \theta_d) + B_d^\intercal W_{d-1}^{-1}\phi(B_dW_d\theta_d)\right],\quad 0 < d<\dmax(X),\\
  \frac{d}{dt} \theta_{D} &= \omega_{D} - B_{D}^\intercal W_{D-1}^{-1}\phi(B_{D}W_{D}\theta_{D})\quad (D:=\dmax(X)).
  \end{split}
\end{equation}

\begin{definition}
When we choose $\omega_d \equiv 0$ for all $d$, we say that~\eqref{eq:nl} is {\bf homogeneous}, and otherwise we say that the system is {\bf inhomogeneous}.  We will refer to the $d$th equation in~\eqref{eq:nl} as the {\bf $d$-simplex flow}, and will use the appropriate names as well, i.e. the ``vertex flow'' is the $0$-simplex flow, the ``edge flow'' is the $1$-simplex flow, etc.
\end{definition}

\begin{rem}
The boundary terms $d=0,D$ are a bit different, but if we look at the middle equation and just drop undefined terms, then we obtain the extremal equations as well.  For example, if $d=0$, then $B_0$ is not defined (and clearly $W_{-1}$ as well), but if we just drop that term and plug in $d=0$ into the first term, we recover the $0$th order equation (and similarly, for the $D$th equation).
\end{rem}

Although the right-hand sides above are complicated, we can see that they at least define a vector field on $C_d(X)$:  for example,
\begin{equation*}
  C_d(X)\xrightarrow{B_{d+1}^\intercal} C_{d+1}(X) \xrightarrow{\phi(\cdot)} C_{d+1}(X) \xrightarrow{W_{d+1}}C_{d+1}(X) \xrightarrow{B_{d+1}} C_d(X)\xrightarrow{W_d^{-1}} C_d(X), 
\end{equation*}
and similarly for the other term.

Note that $\theta_d = {\bf 0}$ is a fixed point of~\eqref{eq:nl} when we choose $\omega_d \equiv 0$.  When we linearize around this solution (see Section~\ref{sec:linearization} for more detail) we obtain the linear flow
\begin{equation}\label{eq:l}
  \frac{d}{dt}z_d = -\phi'(0)\left(W_d^{-1}B_{d+1}W_{d+1}B_{d+1}^\intercal + B_d^\intercal W_{d-1}^{-1}B_dW_d\right) z_d,
\end{equation}
(also similarly applying to the $d=0,D$ terms) and this corresponds to the simplicial Laplacian, giving
\begin{equation*}
  \frac{d}{dt}z_d = -\phi'(0)\mc{L}_dz_d.
\end{equation*}
As shown in~\cite{horak2013spectra}, $\mc{L}_d$ is positive semi-definite as long as all of the weights are nonnegative, and therefore the eigenvalues of the linear operator in~\eqref{eq:l} are all non-positive.  As such,~\eqref{eq:l} has no linearly unstable nodes.  We show below in Section~\ref{sec:convexity} that ${\bf 0}$ is dynamically stable.

\begin{rem}
One thing to note about~\eqref{eq:nl} is that it only involves simplices of a particular dimension $d$.  For a given simplicial complex, we have a family of uncoupled equations, one for each dimension.  We discuss below in Section~\ref{sec:outtro} how we might couple across different dimensions.
\end{rem}

We will also consider the case where the simplicial complex is {\bf unweighted}, or, equivalently, the weight on every simplex is chosen to be 1.  In this case,~\eqref{eq:nl} becomes
\begin{equation}\label{eq:unweighted}
  \frac{d}{dt} \theta_d = -\left[B_{d+1}\phi(B_{d+1}^\intercal \theta_d) + B_d^\intercal \phi(B_d\theta_d)\right].
\end{equation}
Linearizing this equation around $\theta_d\equiv {\bf 0}$ gives (a scalar multiple of) the standard $d$th-order Laplacian
\begin{equation}
  \frac{d}{dt}z_d = -\phi'(0) (B_{d+1}B_{d+1}^\intercal + B_d^\intercal B_d)z_d.
\end{equation}

We can also refer to each of these terms separately as an ``up'' Laplacian and a ``down'' Laplacian, following~\cite{horak2013spectra}.  As we see from~\eqref{eq:nl} or~\eqref{eq:unweighted}, the first term on the right-hand side goes ``up'' to dimension $d+1$  and the second term goes ``down'' to dimension $d-1$.

\subsection{Connection to network consensus models}\label{sec:connection}

We consider the following general network dynamical system:
\begin{equation}\label{eq:K}
  \frac{d}{dt} \eta_i = \zeta_i - \sum_{j=1}^n \gamma_{ij} \phi(\eta_j - \eta_i),
\end{equation}
where we assume $\gamma_{ij} = \gamma_{ji}$ (so that the weight is a function only of the pair $\{i,j\}$ and not their ordering).  Let us make the further assumption that $\phi(\cdot)$ is odd. We show here that this fits into the hierarchy~\eqref{eq:nl} in a simple manner.  We point out that if we choose $\phi(\cdot) = \sin(\cdot)$, then~\eqref{eq:K} becomes the oft-studied Kuramoto model on a graph.  The main result of this section is:

\begin{prop}
  Let $G=(V,E)$ be the graph induced by~\eqref{eq:K} above, i.e. $\{i,j\}\in E \iff \gamma_{ij}\neq 0$.  Let $B_1$ be the incidence matrix of $G$, $W_1$ the $\av E \times \av E$ matrix with $(W_1)_{ee} = \gamma_e$, and for now assume that $W_0$ is the $n\times n$ identity matrix. Let $\omega_0$ be the vector $(\zeta_i)_{i=1}^n$. Then~\eqref{eq:K} is the same flow as 
 \begin{equation}\label{eq:B}
   \frac{d}{dt} \theta_0 = \omega_0-B_1 W_1\phi(B_1^\intercal \theta_0).
\end{equation}
and thus fits into the hierarchy~\eqref{eq:nl} at the vertex flow level.   More generally, if we consider a general $W_0$ with $(W_0)_{vv} = \tau_v$, then this gives a componentwise time-rescaling of~\eqref{eq:K}.
\end{prop}

\begin{rem}
In short, the vertex flow for any simplicial complex gives rise to a network consensus model; conversely, for any network consensus model, there is a simplicial complex (even a graph) whose vertex flow corresponds to it.
\end{rem}

\begin{proof}
Let us consider the graph $G=(V,E)$ with vertices $V = \{v_1,\dots, v_n\}$ and edges $E =\{(i,j): \gamma_{ij}\neq 0\}$.  For each edge $e\in E$, let us choose an arbitrary orientation (e.g. we can say $i\to j$ iff $i<j$ whenever $\{i,j\}\in E$).   The incidence matrix $B_1$ is the the $\av V\times \av E$ matrix given by
\begin{equation*}
  (B_1)_{ve} = \begin{cases} 1,& e= (w,v)\\ -1,& e = (v,w),\\0,&\mbox{else.}\end{cases}
\end{equation*}
We write this more compactly as follows:  let $s(e)$ and $t(e)$ be the source and target of an edge (so, for example, $s(i\to j) = i$ and $t(i\to j) = j$).  Then $(B_1)_{ve} = \delta_{v,t(e)} - \delta_{v,s(e)}$. 
Let us compute the $v$th component of the right-hand side of~\eqref{eq:B}.  We have 
\begin{align*}
  -(B_1 W_1\phi(B_1^\intercal \theta_0))_v
  &= - \sum_{e\in E}(B_1)_{ve} (W_1)_{ee} \phi \left(\sum_{w\in V}(B_1^\intercal)_{ew}\theta_{0,w}\right).
\end{align*}
Since $(B_1^\intercal)_{ew} = (B_1)_{we} = \delta_{w,t(e)}-\delta_{w,s(e)}$, the inner term is $\phi(\theta_{0,t(e)} - \theta_{0,s(e)})$.  Thus we are left with
\begin{equation*}
  - \sum_{e\in E}(B_1)_{ve} \gamma_e \phi(\theta_{0,t(e)} - \theta_{0,s(e)}).
\end{equation*}
There are only two types of edges for which we will get a contribution:  either $e = (v,w)$ or $e = (w,v)$.  If $e = (v,w)$, this sum becomes
\begin{equation*}
  -\sum_{w\in E} (-1) \gamma_{(v,w)} \phi(\theta_{0,w}-\theta_{0,v}),
\end{equation*}
whereas if $e = (w,v)$, then we obtain
\begin{equation*}
  -\sum_{w\in E} (+1) \gamma_{(w,v)} \phi(\theta_{0,v}-\theta_{0,w}).
\end{equation*}
Summing these gives
\begin{equation*}
  \sum_{w\in E} (\gamma_{(v,w)} + \gamma_{(w,v)} ) \phi(\theta_{0,w}-\theta_{0,v}).
\end{equation*}
Since we chose an orientation on the edge $\{v,w\}$, only one of these is nonzero and it is equal to the weight on that edge, and thus we have
\begin{equation*}
  \frac{d}{dt}\theta_{0,v} =  \sum_{w\in E} \gamma_{\{v,w\}} \phi(\theta_{0,w}-\theta_{0,v}),
\end{equation*}
which, after a translation of notation, is exactly~\eqref{eq:K}.

\end{proof}

If we think of the matrix $B$ as encoding the operation of taking boundaries of the edges in the graph, then $B$ naturally maps functions on  vertices to functions on  edges and $B^\intercal$ maps functions on edges to functions on vertices, composing this sequence:
\begin{equation*}
  C_0(X)\xrightarrow{B_{1}^\intercal} C_{1}(X) \xrightarrow{\phi(\cdot)} C_{1}(X) \xrightarrow{W_{1}}C_{1}(X) \xrightarrow{B_{1}} C_0(X).
\end{equation*}
It is clear from~\eqref{eq:B} that if $\theta\in \ker B^\intercal$, then $\theta$ is a fixed point for the dynamics.  Since each row of $B^\intercal$ has exactly one $+1$ and one $-1$, it is clear that any constant vector lies in the kernel.  Moreover, any vector that is constant on a connected component of the graph will lie in $\ker B^\intercal$.  We can also see from inspection that any such vector must be a fixed point of~\eqref{eq:K} by plugging into the right-hand side. 

It is also insightful to note that if we consider a fixed point $\theta\in \ker B^\intercal$, then the linearization of~\eqref{eq:B} is of the form 
\begin{equation*}
  z' = - B_1W_1B_1^\intercal z = -\mathcal{L}_G z, 
\end{equation*}
where $\mc L_G$ is the (weighted) graph Laplacian of the graph $G = (V,E,W)$.  From this we can establish the index of the Jacobian at such a point using standard graph techniques; in particular, $\mc L_G$ is positive semi-definite if all of the edge weights are nonnegative, and the dimension of its nullspace is one more than the number of connected components of $G$ (which is just a statement about the rank of the 0th homology group).

If we let $\phi(x)= x$, then the right-hand side of~\eqref{eq:B} is the graph Laplacian itself, and these are all of the fixed points of~\eqref{eq:B}.  This is just a restatement that the 0th homology group of $G$ has rank equal to one more than the number of connected components of $G$.

It is natural to ask if every fixed point of~\eqref{eq:B} is of this form, but for nonlinear $\phi$,~\eqref{eq:B} can have many more fixed points.  One major class of examples can be constructed when the coefficients $\gamma$ are circulant (i.e. $\gamma_{ij} = \gamma_{\av{i-j}\mbox{ \scriptsize mod} n}$) are the twist solutions $\theta^{(p)}$ with $\theta^{(p)}_k = 2\pi p k/n$ are also a solution to~\eqref{eq:K}.  They have been called ``twist'' solutions~\cite{Wiley.Strogatz.Girvan.06} since they twist around the unit circle in a regular fashion.  We discuss ``twist-like'' solutions in Section~\ref{sec:twistlike} below that are the analogues of these twist solutions.

\subsection{Interpretation of our model ``in coordinates''}\label{sec:interp}

As we have shown in the previous section, the vertex flow gives a standard network consensus model.  More generally, let us consider what a component of~\eqref{eq:nl} might look like.  For now, let us consider the unweighted model~\eqref{eq:unweighted}:

\begin{equation*}\tag{\ref{eq:unweighted}}
  \frac{d}{dt} \theta_d = -\left[B_{d+1}\phi(B_{d+1}^\intercal \theta_d) + B_d^\intercal \phi(B_d\theta_d)\right].
\end{equation*}
Let $F$ be a $d$-simplex, and we compute what $d/dt \theta_{d,F}$ looks like.  Let us consider the first term:
\begin{equation}\label{eq:up}
  \left(B_{d+1}\phi(B_{d+1}^\intercal \theta_d)\right)_F
  	= \sum_{G\in C_{d+1}(X)} (B_{d+1})_{F,G} \cdot \phi\left(\sum_{H\in C_d(X)} (B_{d+1}^\intercal)_{G,H}\theta_{d,H}\right).
\end{equation}
Now notice that $(B_{d+1}^\intercal)_{F,G}\neq 0$ iff $F\in \partial G$, and $(B_{d+1})_{G,H}\neq 0$ iff $H\in \partial G$.  Therefore the only terms that appear in the $F$th coordinate are those $d$-simplices $H$ for which $F\uppadj H$, i.e. $F,H$ are both in the boundary of some $d+1$-simplex.  Similarly, the second term:
\begin{equation}\label{eq:down}
  (B_d^\intercal \phi(B_d\theta_d))_F = \sum_{G\in C_{d-1}(X)} (B_{d}^\intercal)_{F,G} \cdot \phi\left(\sum_{H\in C_d(X)} (B_{d})_{G,H}\theta_{d,H}\right),
\end{equation}
so $(B_d)_{G,H}\neq 0 \iff G\in\partial H$ and $(B_{d}^\intercal)_{F,G}\neq 0\iff G\in\partial F$, so that the $H$th term appears in the $F$th equation iff $F\lowadj H$.

Let us also notice that the coefficient of $\theta_{d,F}$ in $(d/dt)\theta_{d,F}$ is always (effectively) negative.  To see what we mean, consider the term in~\eqref{eq:up} with $H=F$; notice that the sum is now over $G\in C_{d+1}(X)$ where $F\in \bdy G$, and note here that $ (B_{d+1})_{F,G} = (B_{d+1}^\intercal)_{G,H}$ and so both terms have the same sign.  If, for example, $\phi(\cdot)$ is odd, this means every such term will be of the form $-\phi(-\theta_{d,F})$ or $+\phi(+\theta_{d,F})$ --- and recall that there is a minus sign in front of anything.

In short, we can summarize~\eqref{eq:nl} in words as such:  ``Every term on the right hand side of $(d/dt)\theta_d$ is a linear combination of $\phi$'s applied to a linear combination of $d$-simplices, with the property that $\theta_{d,F}$ always appears in $(d/dt)\theta_{d,F}$ with a minus sign, and only those terms appear that correspond to simplices that are either lower or upper adjacent to $F$.''  From a controls standpoint this makes sense:  while the equations are complex, there is a negative feedback in each component --- and it is apparent from this description that the simplicial complex is what is coupling the terms in the equations.  From this we expect $\theta_d\equiv{\bf 0}$ to be a stable fixed point, and we will show this in more detail and rigor below.

From this perspective, the calculation of the previous section makes sense.  If we are considering a vertex flow, then there is no ``down'' term, and only an ``up'' term.  In this case, two vertices can only be coupled if they are both in the boundary of some 1-simplex, i.e. if they are the two vertices in a boundary of some edge --- which is exactly what we see in~\eqref{eq:K}.   See also section~\ref{sec:maximal} below for related examples.

Finally, note that if we allow for general weights, all that happens is that instead of sums we allow for general linear combinations.

\section{Theory}\label{sec:theory}

In this section we show that~\eqref{eq:nl} can be written in gradient form and use this to deduce local stability around certain fixed points, study some properties of the linearization of~\eqref{eq:nl}, study the energy functional, and finally discuss some properties of multistability. 

\subsection{Gradient flow formulation}\label{sec:gradient}

\begin{lem}\label{lem:gradient}
  Let $A$ be an $m\times n$ matrix, $D$ an $n\times n$ diagonal matrix, $F\colon\R\to\R$, and $x\in \R^n$.   Extend $F$ to a function $F\colon \R^n\to\R^n$ by evaluating componentwise. Finally, define the function 
\begin{equation*}
  E(x) := {\bf 1}^\intercal DF(Ax) = \sum_{i=1}^m DF((Ax)_i).
\end{equation*}
Then
  \begin{equation*}
    \nabla E(x) = A^\intercal D F'(Ax),
  \end{equation*}
  where, as above, we interpret $F'(\cdot)$ as a vector function by applying the scalar function componentwise.
\end{lem}

\begin{proof}
First note that 
\begin{equation*}
  DF((Ax)_i) = D(F(Ax)_i) = D_{ii}(F(Ax)_i)
\end{equation*}
and thus 
\begin{equation*}
    \frac{\partial E}{\partial x_k} = \sum_{i=1}^m D_{ii}\frac{\partial }{\partial x_k}  F((Ax)_i) = \sum_{i=1}^m D_{ii}F'((Ax)_i) A_{ik} = (A^\intercal DF'(Ax))_k.
  \end{equation*} 
Since the $k$th component of the gradient is $\partial E/\partial x_k$, this proves the result.  
\end{proof}

\begin{prop}
Let us choose $F(\cdot)$ to be any antiderivative of $\phi(\cdot)$.  Then define
\begin{equation*}
  E_{d,1}(\theta_d) = \sum_{i\in X_{d+1}} W_{d+1}F(B_{d+1}^\intercal \theta_d)_i, \quad
  E_{d,2}(\theta_d) = \sum_{i\in X_{d-1}} F(B_{d} W_d \theta_d)_i
\end{equation*}
and let $E_d(\theta_d) = -\ip{\omega_d}{W_d\theta_d}+ E_{d,1}(\theta_d) + E_{d,2}(\theta_d)$.  Then~\eqref{eq:nl} is the flow
\begin{equation*}
  \theta_d' = -W_d^{-1} \nabla E_d(\theta_d).
\end{equation*}
(If $d=0,D$ then one of the $E_i$ terms is not defined so then $E$ is just the term that is defined.)  In particular, we have shown that our flow is a (rescaled) downhill gradient flow.
  
\end{prop}

\begin{proof}
  Using Lemma~\ref{lem:gradient}, we have
  \begin{equation*}
    \nabla E_{d,1}(\theta_d) = B_{d+1}W_{d+1}\phi(B_{d+1}^\intercal \theta_d),\quad  \nabla E_{d,2}(\theta_d) = W_dB_d^\intercal\phi(B_dW_d \theta_d),   
  \end{equation*}
  (recalling that $W_d^\intercal = W_d$).  Then we have
  \begin{align*}
    -W_0^{-1} \nabla (-\ip{\omega_0}{W_0\theta_0} + E_{0,1}(\theta_0)) &= \omega_0 + W_0^{-1}B_1W_1\phi(B_1^\intercal \theta_0),\\
    -W_D^{-1} \nabla (-\ip{\omega_D}{W_D\theta_D} + E_{D,2}(\theta_D)) &= \omega_d + W_D^{-1}W_D B_D^\intercal \phi(B_DW_D\theta_D),
  \end{align*}
  and for any $0<d<D$ we have
  \begin{align*}
    -W_d^{-1} &\nabla(-\ip{\omega_d}{W_d\theta_d} + E_{d,1}(\theta_d)+ E_{d,2}(\theta_d)) = \\
    &= \omega_d - W_d^{-1}B_{d+1}W_{d+1}\phi(B_{d+1}^\intercal \theta_d) - W_d^{-1}W_dB_d^\intercal\phi(B_dW_d \theta_d),
  \end{align*}
and this matches the right-hand sides of~\eqref{eq:nl}.
\end{proof}

\subsection{Linearization around solutions}\label{sec:linearization}

In this section we compute the linearization around a generic fixed point of~\eqref{eq:nl} and in the process show that the synchronous solution $\theta_d \equiv {\bf 0}$ is dynamically stable --- in fact we will show this for a broader class of solutions that contain ${\bf 0}$.  We will typically denote fixed points of~\eqref{eq:nl} as $x_d$, contrasting to $\theta_d$ for general solutions, and maintain this convention below.

\begin{thm}\label{thm:linearization}
  If $x_d$ is a fixed point of~\eqref{eq:nl}, then the linearization around $x_d$ is the operator $L_{x_d}\colon \R^{\av{X_d}} \to \R^{\av{X_d}}$ given by the formula:
  \begin{equation*}
    L_{x_d}:=-\left[W_d^{-1}B_{d+1}W_{d+1}\diag(f'(B_{d+1}^\intercal x_d))B_{d+1}^\intercal + B_d^\intercal W_{d-1}^{-1} \diag(f'(B_dW_dx_d))(B_dW_d)\right],
  \end{equation*}
  where $\diag(v)$ is the square matrix with the components of $v$ along the diagonal, and we are interpreting $f'(\cdot)$ componentwise on its argument.
  
   As always, for $d=0,D$ we only consider the one of the two terms that is defined (e.g. for $d=0$ we take only the first term above, and for $d=D$ only the second).

    \end{thm}

\begin{proof}
  Consider a function $\phi\colon \R^n\to \R^n$ defined componentwise from a scalar $\phi\colon\R\to \R$, and compute the Taylor expansion of $\phi(x + \epsilon y)$ in $\epsilon$.  Note that the $i$th component of $\phi(x+\epsilon y)$ is $\phi(x_i + \epsilon y_i)$ and we have 
  \begin{equation*}
    \phi(x_i+\epsilon y_i) = \phi(x_i)  + \epsilon \phi'(x_i)y_i,
  \end{equation*}
  and so in general 
  \begin{equation*}
    \phi(x+\e y) = \phi(x) + \e \phi'(x)\circ y,
  \end{equation*}
  where $\circ$ denotes the Schur product of vectors ($(a\circ b)_i := a_i\cdot b_i$).  But also note that we can write 
  \begin{equation*}
    (a\circ b) = \diag(a)\cdot b,
  \end{equation*}
  where $\diag(a)$ is the diagonal matrix whose components are those of $a$.

  Pick $x_d$ a fixed point of~\eqref{eq:nl} and write $\theta_d = x_d + \epsilon z_d$.  We will compute each term separately. First we have:
  \begin{align*}
  W_d^{-1}B_{d+1}W_{d+1}&\phi(B_{d+1}^\intercal (x_d + \e z_d))	
  = W_d^{-1}B_{d+1}W_{d+1}\phi(B_{d+1}^\intercal x_d + \e B_{d+1}^\intercal z_d)\\
  	&= W_d^{-1}B_{d+1}W_{d+1}\left(\phi(B_{d+1}^\intercal x_d) + \e \phi'(B_{d+1}^\intercal x_d) \circ B_{d+1}^\intercal z_d\right) + O(\e^2)\\
	&= W_d^{-1}B_{d+1}W_{d+1} \phi(B_{d+1}^\intercal x_d) + \e W_d^{-1}B_{d+1}W_{d+1}\phi'(B_{d+1}^\intercal x_d) \circ B_{d+1}^\intercal z_d+ O(\e^2)\\
	&=W_d^{-1}B_{d+1}W_{d+1} \phi(B_{d+1}^\intercal x_d) + \e W_d^{-1}B_{d+1}W_{d+1}\diag(\phi'(B_{d+1}^\intercal x_d)) B_{d+1}^\intercal z_d+ O(\e^2).
  \end{align*}
  
Similarly, we have
\begin{align*}
  B_d^\intercal W_{d-1}^{-1}&\phi(B_dW_d\theta_d)
   = B_d^\intercal W_{d-1}^{-1}\phi(B_dW_d(x_d  +\e z_d))\\
   &=B_d^\intercal W_{d-1}^{-1}\phi(B_dW_d x_d  +\e B_dW_d z_d))\\
   &= B_d^\intercal W_{d-1}^{-1}\phi(B_dW_d x_d) + \e B_d^\intercal W_{d-1}^{-1}\phi'(B_dW_d x_d)\circ B_dW_dz_d+ O(\e^2)\\
   &=B_d^\intercal W_{d-1}^{-1}\phi(B_dW_d x_d) + \e B_d^\intercal W_{d-1}^{-1}\diag(\phi'(B_dW_d x_d)) B_dW_dz_d+ O(\e^2)\\
\end{align*}
In each of these equations, the first term is the nonlinear term evaluated at $x_d$, and the $O(\e)$ term will give the linearization.  Note that in each equation the $O(\e)$ term is a linear operator acting on $z_d$, and this linear operator matches the one in the statement above.   As always, for $d=0,D$ we only consider one of these terms.
   
\end{proof}

\begin{cor}
For the homogeneous system (i.e. if we choose $\omega_d=0$ in~\eqref{eq:nl}), and any vector $x_d\in\ker B_{d+1}^{\intercal} \cap \ker B_d$, then $L_{x_d}$ is the same as in~\eqref{eq:l}: a negative scalar multiple of the combinatorial Laplacian, and is thus negative semidefinite.

\end{cor}

\begin{proof}
  If $x_d\in\ker B_{d+1}^{\intercal} \cap \ker B_d$, then both $\diag(f'(B_{d+1}^\intercal x_d))$ and $\diag(f'(B_dW_dx_d))$ are scalar multiples of the identity with multiplier $\phi'(0)$.  Therefore
\begin{align*}
  L_{x_d} &= -\left[W_d^{-1}B_{d+1}W_{d+1}\phi'(0)I_{\av{X_{d+1}}}B_{d+1}^\intercal + B_d^\intercal W_{d-1}^{-1} \phi'(0)I_{\av{X_{d-1}}}(B_dW_d)\right]\\
  	&= -\phi'(0) \left[W_d^{-1}B_{d+1}W_{d+1}B_{d+1}^\intercal + B_d^\intercal W_{d-1}^{-1} B_dW_d\right] \\&= -\phi'(0)\mc L_d.
\end{align*}
  
\end{proof}

\begin{definition}
  Solutions $x_d\in\ker B_{d+1}^{\intercal} \cap \ker B_dW_d$ will be referred to as {\bf homological} solutions, since they form (one choice of) a basis for $H_d(X)$.   
\end{definition}

To motivate the name for these solutions:  recall that $H_d(X)$ is defined as the quotient $\ker\partial_d/\im\partial_{d+1}$.  Now consider $y_d\in \ker B_{d+1}^{\intercal} \cap \ker B_d$.  Using the rank-nullity theorem, $\ker B_{d+1}^\intercal = (\im (B_{d+1}))^\perp$.  Thus we can decompose $\ker B_d \cong \im B_{d+1}\oplus (\im B_{d+1})^\perp$ and we see that we can use elements of $\ker B_{d+1}^{\intercal} \cap \ker B_d$ as a basis for $H_d(X)$.  Of course, in the general case, as we have defined above, the homological solutions live in $\ker B_{d+1}^{\intercal} \cap \ker B_dW_d$, but since all weights are positive, this won't change the dimensions of the set (just possibly rescale some of the entries).

\begin{cor}\label{cor:constant}
  Let $x_d$ be a solution with the property that $f'(B_{d+1}^{\intercal}x_d)$ and $f'(B_d W_d x_d)$ are both constant vectors with positive entries.  Then $L_{x_d}$ is negative semi-definite.
\end{cor}

\begin{proof}
This follows pretty directly from Theorem~\ref{thm:linearization}, since we have
\begin{equation*}
  L_{x_d}  = - \alpha W_d^{-1}B_{d+1}W_{d+1}B_{d+1}^\intercal - \beta B_d^\intercal W_{d-1}^{-1} (B_dW_d),
\end{equation*}
for some $\alpha,\beta > 0$, and each of these individual operators are positive semidefinite by~\cite{horak2013spectra}.
\end{proof}

\begin{rem}
However, this operator is typically singular, and might have multiple zero eigenvalues.  So while there are no linearly unstable nodes, it is not clear at this point that the flow is asymptotically stable.  We address that in the next section.
\end{rem}

\subsection{Energy, convexity, and stability}\label{sec:convexity}

Let $F(\cdot)$ be the antiderivative of $\phi(\cdot)$ with $F(0)=0$.  We always assume that $\phi(0) = 0$ and $\phi'(0)>0$, and therefore $F(0) = F'(0) = 0$ and $F''(0)>0$.  In particular, $F$ is locally strictly convex at 0.

Consider what the energy functional looks like in a neighborhood of ${\bf 0}\in C_d(X)$. First note that in general there will be ``flat'' directions given by the subspace of homological solutions, but otherwise the energy must grow away from ${\bf 0}$.  More precisely, let $z\in C_d(X)$ be decomposed as $z = h + p$, where $h$ is a homological solution and $p$ is orthogonal to the set of homological solutions.  Then either the vectors $B_{d+1}^\intercal z$ or $B_dW_dz$ will have some nonzero entries, or $p=0$.  Since $F$ is locally strictly convex, this implies that for some $\e \ll 1$, $E(\e z) >0$ or $z$ is itself a homological solution.  In particular, while $E(\cdot)$ might be weakly convex in a neighborhood of ${\bf 0}$, it is infinitely flat in the subspace corresponding to homological solutions, and is strictly increasing in the orthogonal direction.  It follows that if we choose as initial condition any perturbation of ${\bf 0}$, it will relax to the subspace of homological solutions exponentially fast.

Let us emphasize that this analysis works only for ${\bf 0}$ (or, more generally, homological solutions) using convexity.  If we consider a general fixed point $x_d\in C_d(X)$, then there is no guarantee that  a fixed point  will be stable, since we the energy is no longer a sum of convex functions.

\subsection{Existence, or lack thereof, of attracting fixed points}\label{sec:compactness}

We have given examples above of simplicial complexes and choices of $\omega_d$ such that~\eqref{eq:nl} has an attracting fixed point.  A natural question to ask is whether there are cases where there are no fixed points at all.

First  note that it is possible to choose any $x_d$ to be a fixed point of~\eqref{eq:nl}, with a judicious choice of $\omega_d$.  Specifically, choose any $x_d\in C_d(X)$, and let 
\begin{equation*}
  \omega_d := W_d^{-1}B_{d+1}W_{d+1}\phi(B_{d+1}^\intercal x_d) + B_d^\intercal W_{d-1}^{-1}\phi(B_dW_dx_d),
\end{equation*}
(with the obvious modification if $d=0,D$), and then $x_d$ is a fixed point of~\eqref{eq:nl}.  We also see that the converse is true:  there exists a fixed point for~\eqref{eq:nl} iff $\omega_d$ is in the range of the operator
\begin{equation}\label{eq:nlop}
  W_d^{-1}B_{d+1}W_{d+1}\phi(B_{d+1}^\intercal \cdot) + B_d^\intercal W_{d-1}^{-1}\phi(B_dW_d\cdot)
\end{equation}

We now see that there are choices of $\omega_d$ so that~\eqref{eq:nl} has no fixed point in many settings.   

\begin{prop}
  Assume that $\phi(\cdot)$ is a continuous and bounded function.  Then the range of~\eqref{eq:nlop} 
is a compact subset of $C_d(X)$.
\end{prop}

(To prove this, note that the range of~\eqref{eq:nlop} is contained in a Cartesian product of closed intervals.)  Clearly, if the range of~\eqref{eq:nlop} is compact, then the set of $\omega_d$ that give a fixed point is also compact, and thus most $\omega_d$ do not give rise to fixed points.    One example commonly considered in network models is the Kuramoto model with $\phi(\cdot) = \sin(\cdot)$.  It is natural to ask whether one can characterize the set of $\omega_d$ that give stable fixed points for~\eqref{eq:nl}, but we expect this question to be challenging, as it is even for the Kuramoto model~\cite{Bronski.DeVille.Park.12, Bronski.DeVille.Ferguson.16, Bronski.Carty.DeVille.17, Bronski.Ferguson.18}.

\section{Examples}\label{sec:examples}

\subsection{Twist-like solutions and multistability}\label{sec:twistlike}

In this section we describe a class of solutions that are not homological solutions, and give examples of complexes on which such solutions can be dynamically stable.  One way to construct such solutions can be seen at the $1$-simplex level, for example.   We will refer to these solutions as ``twist-like'' solutions in analogy to the well-known twist solutions for the Kuramoto system.  

To understand the twist solutions, let us consider the vertex flow in~\eqref{eq:K} with $\phi(\cdot)$ chosen to be $\sin(\cdot)$, namely
\begin{equation}\label{eq:Ksin}
  \frac{d}{dt}\eta_i = \zeta_i + \sum_{j=1}^n \gamma_{ij} \sin(\theta_j-\theta_i).
\end{equation}
Now assume that the weights $\gamma_{ij}$  are circulant; more specifically, let $V = \{0,\dots,n-1\}$ and assume $\gamma_{ij}$ is a function only of $\av{i-j}$.  One sees directly that every configuration of the form $\zeta^{(p)}$ where $\zeta^{(p)}_k = 2\pi pk/n$ is a fixed point for~\eqref{eq:Ksin} (and, in fact, this will be true for any odd $\phi(\cdot)$.  It is natural to ask which choices of $\gamma$ and $p$ make the twist solution stable.  This question was originally raised in the seminal work of~\cite{Wiley.Strogatz.Girvan.06}, and there have been a number of papers expanding on this question in recent years~\cite{DeVille.12, Dorfler.Chertkov.Bullo.13, medvedev2014small, Mehta.etal.14, Mehta.etal.15, DeVille.Ermentrout.16, Delabays.Coletta.Jacquod.16, Delabays.Coletta.Jacquod.17, Ferguson.18, townsend2020dense}, and, not surprisingly, is related to the stability properties of the graph Laplacian. 

One concrete example we can consider is a case considered in~\cite{Wiley.Strogatz.Girvan.06}; assume that for some $k < \lfloor n/2\rfloor$, we have $\gamma_{ij} = 1$ if $\av{i-j} \le k$ and $0$ otherwise.  The components of $B_1^\intercal\theta^{(p)}$ are of the form $\theta^{(p)}_i - \theta^{(p)}_j$, so that $B_1^\intercal\theta^{(p)}$  has at most $2k$ distinct values of a simple form: for example, if $k=1$ then all components of $B_1^\intercal\theta^{(p)}$ are either $2\pi (p/n)$ or $2\pi(1-p/n)$; if $k=2$ then they are all of the form $2\pi(p/n), 2\pi(2p/n), 2\pi(1-p/n), 2\pi(1-2p/n)$, etc.  We also note that $\sin(2\pi x) = \sin(2\pi (1-x))$, so in fact $\phi(B_1^\intercal\theta^{(p)})$ will only take $k$ distinct values, independent of $n$.

This structure allows us a lot of power in solving for fixed points and demonstrating their stability, using Corollary~\ref{cor:constant}.  One concrete example that we can develop is as follows:  let us consider the $2$-complex $X$ with $X_0 = \{0,1,\dots, n-1\}$, $X_1$ is the edges of the form $(i,j)$ with $j=i+1$, and $X_2$ is made up of all triangles of the form $0\le i < j < k < n$.  Let us assume that all weights are one throughout.  Let us find a solution to the simultaneous system
\begin{equation}\label{eq:B1B2first}
  B_1 x_1 = (\alpha,\alpha,\dots, \alpha, \alpha-2\pi),\quad B_2^\intercal x_1 = {\bf 0}.
\end{equation}
Note that if such a solution exists, we then have that $f(B_1 x_1) = f(\alpha){\bf 1}$ and note that ${\bf 1}\in \ker B_1$.  Therefore $x_1$ is a fixed point for the edge flow.  Moreover, note that we have $f'(B_1 x_1) = \cos(\alpha){\bf 1}$ and $f'(B_2^\intercal x_1) = {\bf 1}$.  Therefore, using Corollary~\ref{cor:constant}, we have that $x_1$ is stable iff $\cos(\alpha) >0$.

We show now that for any $n$, there exists a solution of this type with $\alpha = 2\pi/n$, and thus for $n>4$ we have obtained a stable twist-like solution.  Let us choose $\xi$ so that $\xi_{i,j} = 1/2(n-i-j)(i-j)$.  First note that
\begin{equation*}
  (B_1W_1{\xi})_{i,i+1} = \begin{cases}  {\xi}_{i-1,i}-{\xi}_{i,i+1},& 0<i<n-1,\\-{\xi}_{0,n-1}-{\xi}_{0,1},&i=0,\\{\xi}_{n-2,n-1} + {\xi}_{0,n-1},&i=n-1.\end{cases}
\end{equation*}
(Basically the two bottom cases come from the fact that ${\xi}_{n-1,0}$ is not represented in the complex, but if we think of it as $-{\xi}_{0,n-1}$ then the formulas follow a simple rule.)  We then check that
\begin{align*}
  {\xi}_{i-1,i}-{\xi}_{i,i+1} &= \frac12(n-2i+1)-\frac12(n-2i-1) = 1,\\
  -{\xi}_{0,n-1}-{\xi}_{0,1} &= -\frac{n-1}2-\frac{n-1}2 = 1-n,\quad
  {\xi}_{n-2,n-1} + {\xi}_{0,n-1} = \frac{3-n}2 + \frac{n-1}2 = 1.
\end{align*}
From this we see that 
\begin{equation*}
  (B_1W_1{\xi}) = (1,1,\dots, 1-n).
\end{equation*}
Similarly, we see that the value of $B_2^\intercal\xi$ on the triangle $f_{i,j,k}$ is
\begin{align*}
  -\xi_{i,j} + \xi_{i,k} - \xi_{j,k} 
  &= -\frac12\left((n-i-j)(i-j) - (n-j-k)(j-k)+(n-i-k)(i-k)\right)\\
  &=-\frac12\left(n(i-j) - (i^2-j^2) - (n(i-k) - (i^2-k^2)) + (n(j-k) - (j^2-k^2)\right)\\
  &= -\frac12\left(n(i-j-i+k_j-k) - (i^2-j^2-i^2+k^2+j^2-k^2)\right) = 0.
\end{align*}
Putting this together we have
\begin{equation*}
  B_2^\intercal \xi = {\bf 0},\quad (B_1W_1\xi) = (1,1,\dots, 1-n).
\end{equation*}
Choosing $x_1 = \alpha \xi$ with $\alpha = 2\pi/n$ gives
\begin{equation*}
  B_2^\intercal x_1 = {\bf 0},\quad (B_1W_1x_1) = (\alpha,\alpha,\dots,\alpha-2\pi).
\end{equation*}

Choosing the last component to be the different one in~\eqref{eq:B1B2first} was arbitrary, but it is apparent that there is nothing special about choosing the last component; the system is clearly invariant under any $n$-cycle permutation and we could choose the ``jump'' to be wherever we like.

\subsection{Maximal dimensional flows}\label{sec:maximal}

Here we are considering the $d=D$ flow
\begin{equation*}
\frac{d}{dt} \theta_{D} = \omega_{D} - B_{D}^\intercal W_{D-1}^{-1}\phi(B_{D}W_{D}\theta_{D})\quad (D:=\dmax(X)).
\end{equation*}

This equation is analogous to the vertex flow because there is only one term on the right-hand side.  Each coordinate $(d/dt)\theta_{D,F}$ will only have terms of the form $\theta_{D,H}$ where $F$ and $H$ are lower adjacent.

Let us first consider the case where $D=1$; namely, our simplicial complex is simply a graph, and we consider the edge flow on that graph. Let us write the vertex weights as $\tau_v$ and the edge weights as $\gamma_e$.

 Recall that $(B_1)_{ve} = \delta_{v,t(e)} - \delta_{v,s(e)}$, and thus we have
\begin{align*}
  -(B_1^\intercal &W_0^{-1} \phi(B_1W_1\theta_1))_e
  	= -\sum_{v\in X_0} (B_1^\intercal)_{ev} \tau_v^{-1} \phi\left(\sum_{e'\in X_1} (B_1)_{ve'} \gamma_{e'}\theta_{1,e'}\right)\\
	&= -\sum_{v\in X_0} (\delta_{v,t(e)}-\delta_{v,s(e)}) \tau_v^{-1} \phi\left(\sum_{e'\in X_1} (\delta_{v,t(e')}-\delta_{v,s(e')}) \gamma_{e'}\theta_{1,e'}\right)\\
	&=-\sum_{v\in X_0} (\delta_{v,t(e)}-\delta_{v,s(e)}) \tau_v^{-1}\phi\left(\sum_{e'\in t^{-1}(v)}\gamma_{e'}\theta_{1,e'} - \sum_{e'\in s^{-1}(v)}\gamma_{e'}\theta_{1,e'}\right)\\
	&= \frac{1}{\tau_{s(e)}}\phi\left(\sum_{e'\in t^{-1}(s(e))}\gamma_{e'}\theta_{1,e'} - \sum_{e'\in s^{-1}(s(e))}\gamma_{e'}\theta_{1,e'}\right)\\&\quad-\frac{1}{\tau_{t(e)}}\phi\left(\sum_{e'\in t^{-1}(t(e))}\gamma_{e'}\theta_{1,e'} - \sum_{e'\in s^{-1}(t(e))}\gamma_{e'}\theta_{1,e'}\right).
\end{align*}

This formula comes out pretty nicely (once the chalk has cleared, at least).  Basically as we see, the only terms that appear on the right-hand side of $(d/dt)\theta_{1,e}$ are those edges which share a vertex with $e$.  Moreover, we see that if the graph is regular (every vertex has the same degree $\deg$) then each sum inside a nonlinearly will have the same number of terms.  Note, however, except for the case where $\deg = 2$, the edge flow is absolutely not equivalent to a vertex flow on a graph.

\subsection{Triangulation of the torus}\label{sec:torus}

Let us derive all three equations on a simplicial complex representing a triangulation of the torus.  We consider the homogeneous equation $(\omega_d\equiv 0)$ for simplicity.

Choose $m,n\ge 2$.  We define the vertex set as $\{(i,j),0\le i\le m, 0\le j \le n\}$.  The edges of the complex are given by $\eh i j, \ev i j, \ed i j$ where 
\begin{align*}
  \eh i j = [(i,j),(i,j+1)],\mbox{ i.e. }\quad (i,j) &\xrightarrow{\eh i j} (i, (j+1)\mod n)\\
    \ev i j = [(i,j),(i+1,j)],\mbox{ i.e. }(i,j) &\xrightarrow{\ev i j} ((i+1) \mod m, j)\\
     \ed i j = [(i,j),(i+1,j+1)],\mbox{ i.e. } (i,j) &\xrightarrow{\ed i j} ((i + 1) \mod m, (j+1)\mod n)
\end{align*}
and the triangles of the complex are given by $\sil i j, \sir i j$, where
\begin{equation*}
  \sil i j = [(i,j), (i+1,j), (i+1,j+1)],\quad \sir i j =  [(i,j), (i,j+1), (i+1,j+1)].
\end{equation*}
The fact that we take mods in the coordinates is what makes this to be a torus.  (This is  a generalization of~\cite[Example 3, p. 17]{Munkres.book}, for example.)

\begin{figure}[th]
\begin{center}
\scalebox{1.25}{%
\begin{tikzpicture}[->,>=stealth',shorten >=1pt,auto,node distance=3cm,
  thick,main node/.style={circle,fill=yellow!30,draw,font=\sffamily\bfseries},scale=1]

  \draw [-,fill=red!10] (0,0) -- (0,-3) -- (3,-3);
  \draw [-,fill=red!10] (3,0) -- (3,-3) -- (6,-3);
  \draw [-,fill=red!10] (0,-3) -- (0,-6) -- (3,-6);
  \draw [-,fill=red!10] (3,-3) -- (3,-6) -- (6,-6);
  \draw [-,fill=green!10] (0,0) -- (3,0) -- (3,-3);
  \draw [-,fill=green!10] (3,0) -- (6,0) -- (6,-3);
  \draw [-,fill=green!10] (0,-3) -- (3,-3) -- (3,-6);
  \draw [-,fill=green!10] (3,-3) -- (6,-3) -- (6,-6);

  \node[main node][label=above:{$v_{i-1,j-1}$}] (00) {};
  \node[main node][label=left:{$v_{i,j-1}$}] (10) [below of=00] {};
  \node[main node][label=left:{$v_{i+1,j-1}$}] (20) [below of=10] {};
  \node[main node][label=above:{$v_{i-1,j}$}] (01) [right of=00] {};
  \node[main node][label=above right:{$v_{i,j}$}] (11) [below of=01] {};
  \node[main node][label=below:{$v_{i+1,j}$}] (21) [below of=11] {};
  \node[main node][label=above:{$v_{i-1,j+1}$}] (02) [right of=01] {};
  \node[main node][label=above right:{$v_{i,j+1}$}] (12) [below of=02] {};
  \node[main node][label=above right:{$v_{i+1,j+1}$}] (22) [below of=12] {};
  
  \node[scale=0.75,color=red] at (1,-2.25) {$\sil{i-1}{j-1}$};
  \node[scale=0.75,color=red] at (4,-2.25) {$\sil{i-1}{j}$};
  \node[scale=0.75,color=red] at (1,-5.25) {$\sil{i}{j-1}$};
  \node[scale=0.75,color=red] at (4,-5.25) {$\sil{i}{j}$};
  \node[scale=0.75,color=green!30!black] at (2.25,-0.75) {$\sir{i-1}{j-1}$};
  \node[scale=0.75,color=green!30!black] at (5.25,-0.75) {$\sir{i-1}{j}$};
  \node[scale=0.75,color=green!30!black] at (2.25,-3.75) {$\sir{i}{j-1}$};
  \node[scale=0.75,color=green!30!black] at (5.25,-3.75) {$\sir{i}{j}$};

  \path[every node/.style={font=\sffamily\small},color=blue!55!black]
   (00) edge node {$\eh {i-1}{j-1}$} (01)
	 (01) edge node {$\eh {i-1}{j}$} (02)
	 (10) edge node [below] {$\eh {i} {j-1}$} (11)
	 (11) edge node [below] {$\eh {i}{j}$} (12)
	 (20) edge node [below] {$\eh {i+1} {j-1}$} (21)
	 (21) edge node [below] {$\eh {i+1}{j}$} (22)
	 (00) edge node [rotate=-35]{$\ed {i-1}{j-1}$} (11)
	 (01) edge node [rotate=-35]{$\ed {i-1} {j}$} (12)
	 (10) edge node [rotate=-35]{$\ed {i}{j-1}$} (21)
	 (11) edge node [rotate=-35]{$\ed {i} {j}$} (22)
	 (00) edge node [left] {$\ev {i-1} {j-1}$} (10)
	 (01) edge node {$\ev {i-1} {j}$} (11)
	 (02) edge node {$\ev {i-1} {j+1}$} (12)
	 (10) edge node [left] {$\ev {i} {j-1}$} (20)
	 (11) edge node {$\ev {i} {j}$} (21)
	 (12) edge node {$\ev {i} {j+1}$} (22);
	 
	 \end{tikzpicture}
}

\end{center}
\caption{A piece of the simplicial complex triangulating a torus, with vertices, edges, and triangles labeled.  We use the convention that the edges are blue(ish), the ``left'' triangles red, and the ``right'' triangles green --- and the labels match this coloring scheme as well. Recall, as described in the text, that in the full grid (not pictured) the left/right and up/down edges are identified.  }
\label{fig:torus}
\end{figure}
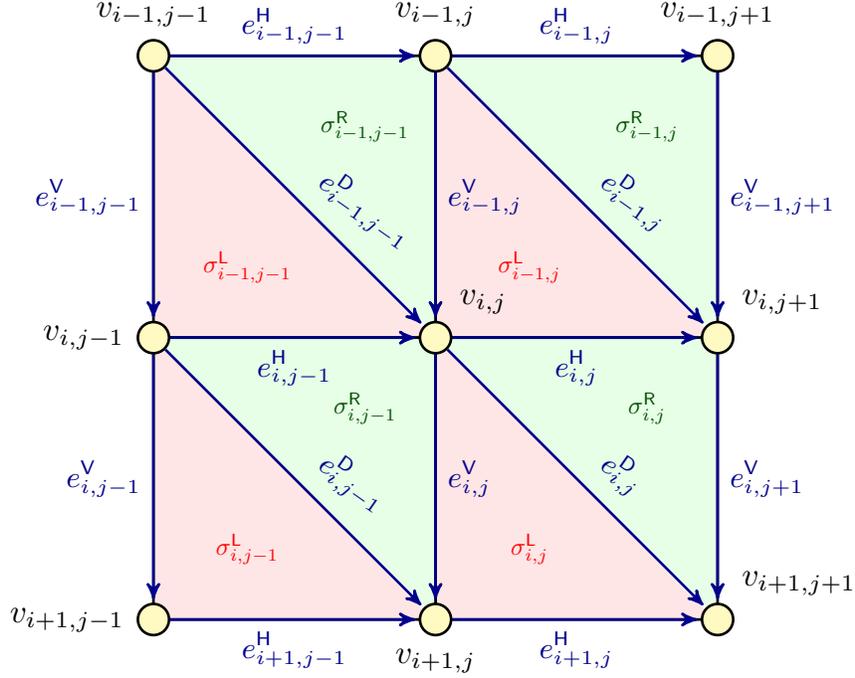

We  compute that 
\begin{align*}
 \partial_2 \sil i j & = \ed i j - \ev i j - \eh{i+1}j, &\partial_1 \eh i j &= v_{i,j+1} - v_{i,j}\\
 \partial_2 \sir i j & = \eh i j - \ev i {j+1} - \ed{i}j, &\partial_1 \ev i j &= v_{i+1,j} - v_{i,j}\\
 &&\partial_1 \ed i j &= v_{i+1,j+1} - v_{i,j}
\end{align*}
where we take modulos in the appropriate places.  The main computation here is to get the 0-, 1-, and 2-flows on this simplicial complex, using~\eqref{eq:nl}.  Expanding out this formula, we obtain
\begin{align}
  \frac{d}{dt}\theta_2 &= -B_2^\intercal \sin(B_2\theta_2),\label{eq:torus2}\\
  \frac{d}{dt}\theta_1 &= -B_2\sin(B_2^\intercal\theta_1) - B_1^\intercal \sin(B_1\theta_1),\label{eq:torus1}\\
  \frac{d}{dt}\theta_0 &= -B_1 \sin(B_1^\intercal \theta_0),\label{eq:torus0}
\end{align}

Then the vertex flow in coordinates is:
\begin{align*}
  \frac{d}{dt}(\theta_0)_{(i,j)} &=\sin(\theta_{0,i,j-1}-\theta_{0,i,j}) + \sin(\theta_{0,i,j+1}-\theta_{0,i,j}) + \sin(\theta_{0,i-1,j}-\theta_{0,i,j})+\\
	&\quad + \sin(\theta_{0,i+1,j}-\theta_{0,i,j}) + \sin(\theta_{0,i-1,j-1}-\theta_{0,i,j}) - \sin(\theta_{0,i+1,j+1}-\theta_{0,i,j})
\end{align*}
We can read this off of Figure~\ref{fig:torus} as follows:  if we look at the node $v_{ij}$, note that it has six neighbors: $v_{i-1,j}$, $v_{i,j-1}$, $v_{i+1,j}$, $v_{i+1,j}$, $v_{i-1,j-1}$, $v_{i+1,j+1}$, and each of these corresponds to a single term on the right-hand side of the vertex flow.

The edge flow in coordinates is:
\begin{align*}
  \frac{d}{dt}\theta_{\sv i j H}
  &= \sin(\theta_{\sv i{j-1}H} + \theta_{\sv{i-1}j V}+\theta_{\sv{i-1}{j-1}D} - \theta_{\sv i j H}- \theta_{\sv i j V}- \theta_{\sv i j D})\\
  &\quad + \sin(-\theta_{\sv i j H} - \theta_{\sv{i-1}{j+1} V}-\theta_{\sv{i-1}{j}D} + \theta_{\sv i{j+1}H}+\theta_{\sv i{j+1} V}+ \theta_{\sv i{j+1} D})\\
  &\quad + \sin(-\theta_{\sv i j H} - \theta_{\sv {i-1}j V} + \theta_{\sv {i-1}j D}) + \sin(-\theta_{\sv i j H}-\theta_{\sv i{j+1}V} + \theta_{\sv i j D}),\\
   \frac{d}{dt}\theta_{\sv i j V}
  &= \sin(-\tsv i j V -\tsv i j D -\tsv i j H + \tsv{i}{j-1}H + \tsv{i-1}{j-1}D+\tsv{i-1}jV) \\
  &\quad + \sin(-\tsv i j V - \tsv i{j-1}D-\tsv {i+1}{j-1}H + \tsv {i+1}j H + \tsv{i+1}jD+ \tsv{i+1}jV)\\
  &\quad + \sin(-\tsv i j V - \tsv {i+1}j H + \tsv i j D) + \sin(-\tsv i j V - \tsv i{j-1}H + \tsv i{j-1}D),\\
  \frac{d}{dt}\tsv i j D
  &= \sin(-\tsv i j D-\tsv i j V-\tsv i j H+\tsv i{j-1}H+\tsv{i-1}{j-1}D + \tsv{i-1}j V)\\
  &\quad + \sin(-\tsv i j D - \tsv i{j+1}V - \tsv{i+1}j H + \tsv{i+1}{j+1}H+ \tsv{i+1}{j+1}V+ \tsv{i+1}{j+1}D)\\
  &\quad + \sin(-\tsv i j D + \tsv i{j+1}V + \tsv i j H) + \sin(-\tsv i j D + \tsv i j V + \tsv{i+1}jH).  
\end{align*}

Let us consider, for example, the $\tsv i j D$ equation.  There are four terms on the right-hand side of this equation, two of degree six and two of degree three.  The first two terms correspond to the down Laplacian, and the latter two correspond to the up Laplacian.  

To understand the first two terms:  consider the edge $\ed i j$, and note that its boundary contains two vertices:  $v_{i,j}$ and $v_{i+1,j+1}$, and each of these vertices have six edges connected to them.  If we for example consider the edges going into or out of $v_{i,j}$, we see that each of these is represented once in the first term on the right-hand side of $(d/dt)\tsv i j D$.  Moreover, the signs can be read off as follows:  the three edges pointing out of the vertex $v_{i,j}$ have negative signs, and the three edges pointing into the vertex have positive signs.  (More generally, those edges with the same orientation with respect to the vertex as a particular edge will appear with a minus sign, and those with the opposite orientation will appear with a plus sign.)   

To understand the latter two terms, note that $\ed ij$ is in the boundary of two triangles:  $\sil ij$ and $\sir ij$.  If we consider $\sir ij$ for example, we see that its boundary contains three edges:  $\ed ij$, $\ev i{j+1}$, and $\eh ij$. Note that since the latter two edges go in the opposite direction as $\ed i j$, they appear with plus signs and only $\ed ij $ appears with a minus sign.

The triangle flow in coordinates is:
\begin{align*}
   \frac{d}{dt}\theta_{2,\sv i j L} &=  \sin(\theta_{2,\sv{i+1}j R}-\theta_{2,\sv i j L}) + \sin(\theta_{2,\sv i{j-1}R}-\theta_{2,\sv i j L}) + \sin(\theta_{2,\sv{i-1}jL}-\theta_{2,\sv i j L})\\
   \frac{d}{dt}\theta_{2,\sv i j R} &= \sin(\theta_{2,\sv{i}j L}-\theta_{2,\sv i j R}) + \sin(\theta_{2,\sv {i-1}{j}L}-\theta_{2,\sv i j R}) + \sin(\theta_{2,\sv{i}{j+1}L}-\theta_{2,\sv i j R})\\
\end{align*}

We note here that the $2$-flow looks like a vertex flow, but this is because each edge in this complex is on the boundary of exactly two triangles --- so there actually is a graph out there for which the vertex flow is equivalent to the triangle flow on this complex.  (This will not be true in general, of course.)  Note that it has three terms, because for any triangle, there are exactly three triangles with which it shares an edge.  For example, for the triangle $\sir ij$, we see that it shares an edges with triangles above, to the right, and to the lower left (all of which are ``left'' triangles), giving $\sil{i-1}j$, $\sil{i}{i+j}$, and $\sil i j$ --- and this corresponds to the three terms on the right-hand side of the equation in an obvious manner.

Now let us consider the fixed points of these equations.  
The vertex flow is just the classical Kuramoto system on the vertices of the induced graph, which is a nice regular graph of degree 6.  Similarly, the triangle flow is also a classical Kuramoto system on a graph of degree $2$.  The new system here is the edge flow.  First note that since $\dim H_1(X) = 2$, the set of homological solutions will be $2$-dimensional.  As such, there is a two-dimensional linear subspace of solutions containing  the origin.  We give an example of such a solution in Figure~\ref{fig:torus-hom} below (note that it is quantitatively complicated, even though it has a simple qualitative description).

\begin{figure}[th]
\begin{center}
\includegraphics[width=0.9\textwidth]{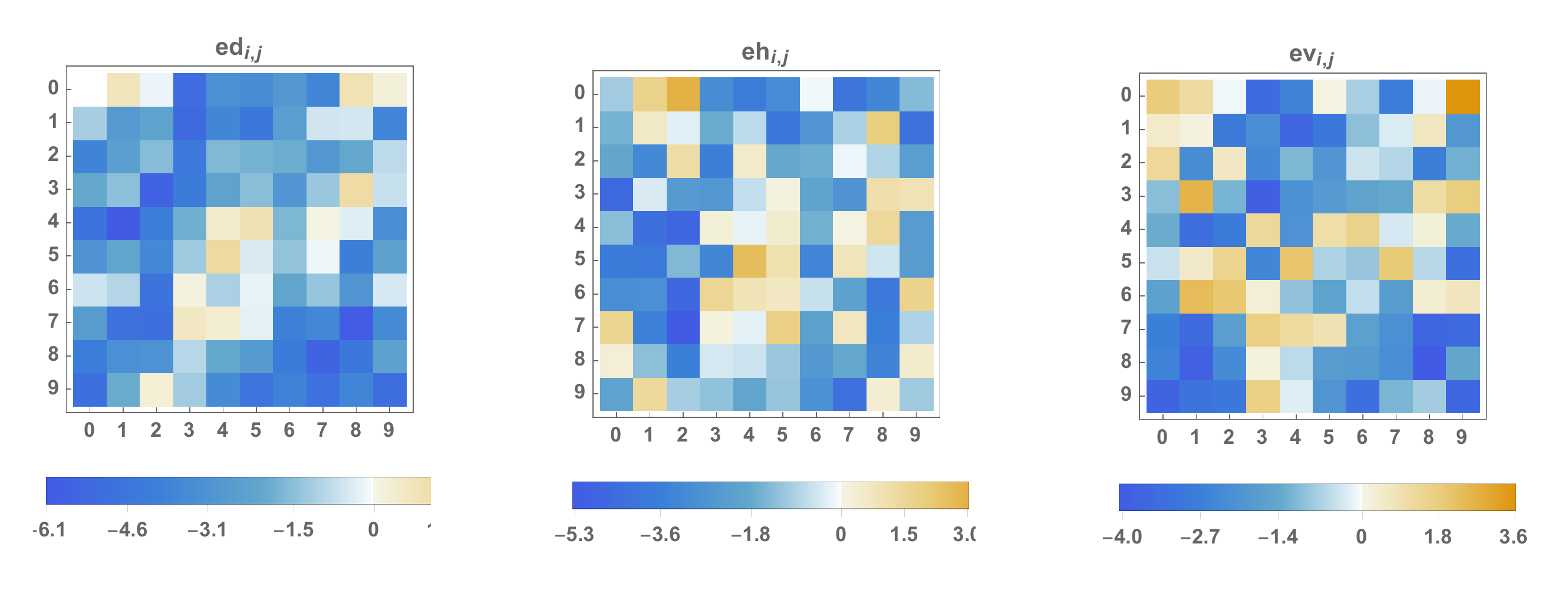}
\end{center}
\caption{Homological solution of the edge flow for the torus triangulation with $m=n=10$ and with the choice $\theta_{0,0,D} = 0, \theta_{0,1,D} = 1$ (recall that there are two degrees of freedom that need to be set).}
\label{fig:torus-hom}
\end{figure}

We can again repeat the construction in Section~\ref{sec:twistlike} as follows:  we find a  solution to the simultaneous system
\begin{equation}\label{eq:B1B2}
  B_1 x_1 = (\alpha-2\pi, \alpha,\alpha,\dots, \alpha),\quad B_2^\intercal x_1 = {\bf 0},
\end{equation}
where the first component in the vector above is the $v_{0,0}$ component (recall that $B_1x_1\in C_0(X)$). We find numerically that choosing $\alpha = 2\pi/(mn)$ gives a solution for a wide variety of $m,n$, and we present such an example numerically in Figure~\ref{fig:torus-hom}.  By Corollary~\ref{cor:constant} as long as $mn>4$ we obtain a stable fixed point.  Of course, just as there is a two-dimensional space of homological solutions, the solution to this family of equations also has two degrees of freedom.  By symmetry, there is nothing special about the $(0,0)$ component above, and so, for example, in~\eqref{eq:B1B2} we could have chosen any other component to be the one equal to $2\pi-\alpha$ and obtained a different twist-like solution.

\section{Conclusions and outlook}\label{sec:outtro}

In this paper we have presented a model of synchronization/consensus that is defined on simplicial complexes and has several nice properties.  In particular, we have shown that this is the gradient flow of a certain energy functional on a simplicial complex, which allows us to establish stability of certain steady states even when the linearization happens to be degenerate.

While we have established some basic results about stability, it is clear that these systems can support a rich collection of steady states.  For example, the type of constructions used to generate nonhomological solutions in Sections~\ref{sec:twistlike} and~\ref{sec:torus} can be generalized to a wide variety of examples.  One concrete question, for example:  for the vertex flow, all that is needed for twist solutions to exist is that the graph weights be circulant.  Is there a simple characterization describing which simplicial complexes will support twist-like solutions?  More generally, it would be nice to have a better understanding of exactly how symmetries in the underlying simplicial complex drives symmetries in the steady states.

We also point out that we have established multistability (the coexistence of multiple steady states) in several examples above.  Although it is beyond the scope of the current work, it would be interesting to understand how often a random initial condition decays to a homological solution in those cases where nonhomological solutions are attracting. 

Another question of interest would be to consider the role of different choices for $\phi(\cdot)$ and how this might impact the dynamics. For example, a common choice of $\phi(\cdot) = \sin(\cdot)$ gives rise to generalizations of the Kuramoto model, and specific properties of the function $\sin(\cdot)$ were important above.  Examples might include if $\phi$ were unbounded but nonlinear (esp. invertible), or if the cardinality of the preimage of $\phi(\cdot)$ was not essentially constant --- in these cases the existence and stability of the steady states would have to be attacked in a different manner.

It is worth saying a few words about what can be done differently with this model as well.  As mentioned in the introduction, the model we study here has  already been presented in the innovative paper~\cite{millan2020explosive}, but in that paper the authors considered a different family of questions than those considered here.  Whereas in this work we seek to understand the stability of certain stationary solutions, and the possibility of multistability for these systems, in~\cite{millan2020explosive} the authors considered the problem of synchronization (and in particular, explosive synchronization) for this model when the forcing terms are drawn from a distribution and the strength of the simplicial coupling is varied.  In particular, in~\cite{millan2020explosive} the authors note and exploit the structure of the Hodge decomposition of the simplicial Laplacian to show that the effective equations for the order parameter decouple nicely into ``up'' and ``down'' terms.  It is plain from these results that the present model is likely to exhibit all of the complexity that has been thoroughly studied in network-level consensus models and as such the present model is certainly worthy of extensive further study.

Finally, we note that in~\eqref{eq:nl}, the flows at each dimension were independent --- as such, it is possible to set initial data so that the solutions at each level are not connected.  It would be interesting to consider modifications of this system that would require the solutions at various dimensions to be related in some manner, so that we would have a flow on the entire simplex, instead of a family of flows, one for each dimension.

%
%\bibliographystyle{plain}
%\bibliography{sk}

\end{document}